\documentclass[12pt,reqno]{amsart}
\usepackage{amssymb}
\usepackage{amsmath}
\usepackage{enumitem}
\usepackage{mathrsfs}
\usepackage[all]{xy}
\usepackage{hyperref}

\usepackage[margin=1.5in]{geometry}

\RequirePackage{mathrsfs} \let\mathcal\mathscr

\newtheorem{theorem}{Theorem}[section]
\newtheorem{lemma}[theorem]{Lemma}

\newtheorem{corollary}[theorem]{Corollary}
\newtheorem{conjecture}[theorem]{Conjecture}

\theoremstyle{definition}
\newtheorem*{ack}{Acknowledgements}
\newtheorem{remark}[theorem]{Remark}
\newtheorem{example}[theorem]{Example}
\newtheorem{definition}[theorem]{Definition}

\numberwithin{equation}{section} \numberwithin{figure}{section}

\DeclareMathOperator{\Pic}{Pic}

 \DeclareMathOperator{\rank}{rank}

 \DeclareMathOperator{\Val}{Val}

 \DeclareMathOperator{\Index}{Index}
\DeclareMathOperator{\Jac}{Jac}

\newcommand{\SL}{\textrm{SL}}

\newcommand{\kbar}{\overline{k}}

\newcommand{\Adele}{\mathbf{A}}
\newcommand{\LL}{\mathcal{L}}

\newcommand\PP{\mathbb{P}}

\newcommand\ZZ{\mathbb{Z}}
\newcommand\NN{\mathbb{N}}
\newcommand\QQ{\mathbb{Q}}
\newcommand\RR{\mathbb{R}}

\newcommand\GG{\mathbb{G}}

\newcommand\OO{\mathcal{O}}

\renewcommand{\leq}{\leqslant}

\renewcommand{\geq}{\geqslant}

\newcommand{\x}{\mathbf{x}}

\newcommand{\bal}{\boldsymbol{\alpha}}

\renewcommand{\d}{\mathrm{d}}

\title{Varieties with too many rational points}

\author{T.D. Browning}
\address{School of Mathematics\\
University of Bristol\\ Bristol\\ BS8 1TW}
\email{t.d.browning@bristol.ac.uk}

\author{D. Loughran}
\address{
School of Mathematics \\
University of Manchester \\
Oxford Road \\
Manchester \\
M13 9PL}
\email{daniel.loughran@manchester.ac.uk}


\begin{document}

\begin{abstract}
	We investigate Fano varieties defined over a number field that
	contain subvarieties whose number of rational points of bounded height is comparable to the total number on the variety.
\end{abstract}

\maketitle
\tableofcontents

\thispagestyle{empty}

\section{Introduction}

This paper is concerned with the growth rate of rational points of bounded height on Fano varieties
defined over a number field, i.e.\ non-singular projective varieties which have an ample anticanonical bundle $\omega_X^{-1}$.
Let $X$ be a Fano variety over a number field $k$ such that
$X(k)$ is dense in $X$ under the Zariski topology,
and let $H$ be an anticanonical height function on $X$.
Then we may define the counting function
$$
N(Z,H,B)=\#\{x\in Z: H(x)\leq B\},
$$
for any $B>0$ and any $Z \subset X(k)$. If $Z \subset X$ is a locally closed subset, we shall write
$N(Z,H,B)$ for $N(Z(k),H,B)$. Manin and his collaborators \cite{FMT89} have formulated a conjecture
on the asymptotic behaviour of such counting functions, as $B \to \infty$.
The Manin conjecture
predicts the existence of an open subset $U\subset X$
such that for any anticanonical height function $H$ on $X$, 
there is
a constant $c_{U,H}>0$
such that
\begin{equation} \label{con}
	N(U,H,B) \sim c_{U,H} B(\log B)^{\rho(X)-1}, \quad \text{as } B \to \infty,
\end{equation}
where $\rho(X)=\rank \Pic(X)$.
In general one needs to look at an open subset in order to
avoid \emph{accumulating subvarieties}.
According to Peyre \cite[D\'ef.~2.4.1]{p251},
a proper irreducible subvariety $Y\subset X$
is said to be {\em strongly accumulating} if there exists a choice
of anticanonical height function $H$ on $X$ such that
for any non-empty
open subset $Y^\circ \subset Y$, there exists a non-empty open subset $U\subset X$ for which
$$
\limsup_{B\rightarrow \infty} \frac{\log N(Y^\circ,H,B)}{\log B}>
\limsup_{B\rightarrow \infty} \frac{\log N(U,H,B)}{\log B}.
$$
The most familiar example comes from non-singular cubic surfaces $X\subset \PP^3$.
If one of the $27$ lines
is defined over $k$, and we denote this line by $Y$, then the left hand side of this inequality is $2$, by work of
Schanuel \cite{Sch79} (see Peyre  \cite[Cor.~6.2.18]{Pey95} for general height functions).
On the other hand, it is expected that the right hand side should be $1$ when $U=X\setminus \{\mbox{27 lines}\}.$

It is known that this conjecture cannot hold for all Fano varieties, as shown by
Batyrev and Tschinkel \cite{BT96} through their analysis of
a certain bundle of cubic surfaces in $\PP^3\times\PP^3$.
Nonetheless the conjecture is expected to hold for a very wide class of Fano varieties.
Moreover, Peyre \cite{Pey95} has conjectured an explicit value $c_{H,\text{Peyre}}$
for the leading constant $c_{U,H}$ appearing in \eqref{con}.
Only very recently have counter-examples to Peyre's conjecture
on the leading constant been found by Le Rudulier \cite{LeR13}.
In this paper we present many more counter-examples to Peyre's conjecture
and describe explicit conditions under which a given
Fano variety provides  a counter-example to Peyre's conjecture.

Our investigation will concentrate on
 Fano varieties  whose  counting functions have the same order of magnitude  as the counting function associated to a proper closed subvariety.
Experimental  investigations into this phenomenon have been carried out  by Elsenhans and Jahnel
 \cite{EJ06} for cubic and quartic threefolds in $\PP^4$,
 and by Elsenhans \cite{elsenhans} for some quadric bundles in $\PP^1\times \PP^3$.

Following  Peyre \cite[D\'ef.~2.4.2]{p251},
a proper irreducible  subvariety $Y\subset X$ is said to be {\em weakly accumulating}
if if there exists a choice of anticanonical height function $H$ on $X$ such that 
for any non-empty open subset $Y^\circ \subset Y$, there exists a non-empty open subset $U\subset X$ for which
$$
\limsup_{B\rightarrow \infty} \frac{N(Y^\circ,H,B)}{N(U,H,B)}>0.
$$
We have found it convenient to refine this definition as follows.

\begin{definition}\label{d}
Let $X$ be a Fano variety over a number field $k$. We shall say that a proper closed
irreducible subvariety  $Y \subset X$ over $k$  is {\em saturated}
if there exists a choice of anticanonical height $H$ on $X$ such that
\begin{equation} \label{def:saturated}
		0<\liminf_{B \to \infty}\frac{N(Y^\circ,H,B)}{B(\log B)^{\rho(X)-1}} < \infty,
\end{equation}
	for every non-empty open subset $Y^\circ\subset Y$.
\end{definition}
Note that standard properties of height functions (see Section  \ref{s:proof})
imply that if $Y \subset X$ is saturated,
then \eqref{def:saturated} holds for every choice of anticanonical height function.
Assuming the truth of \eqref{con}, moreover, it is clear that any
saturated subvariety for $H$ is  weakly accumulating for $H$.
A geometric analogue of this phenomenon  has also been  investigated in recent work of  Hassett, Tanimoto and Tschinkel
\cite{sho}.

Our first result on saturated subvarieties is the following.

\begin{theorem}\label{thm:WA1}
	Let $X$ be a Fano variety over a number field $k$ and let $Y \subset X$ be a saturated
	closed subset over $k$. 
	Then for any open subset $U \subset X$ which meets  $Y$ and any $A \geq 0$, there
    exists a choice of anticanonical height function $H_A$ on $X$ such that
    $$\liminf_{B \to \infty}\frac{N(U,H_A,B)}{B(\log B)^{\rho(X)-1}} > (1 + A) c_{H_A,\textrm{Peyre}}.$$
\end{theorem}

The conclusion of this result is that Manin's conjecture with Peyre's constant cannot hold for such open subsets
with respect to every choice of anticanonical height function.
Note that a similar phenomenon occurs in a paper of
Swinnerton-Dyer \cite[p.~376]{swd94}. Namely, Swinnerton-Dyer remarks that if $X$ fails 
weak approximation then the definition of the leading constant in Manin's conjecture needs to take
this into account, since one can  otherwise modify the leading constant (by changing the local metrics
at the places where weak approximation fails)  without changing the counting function.

From Theorem \ref{thm:WA1}, we obtain the following immediate corollary.

\begin{corollary}\label{cor:WA2}
	Let $X$ be a Fano variety over a number field $k$ and suppose that the union of the
	saturated subvarieties $Y \subset X$ over $k$ is Zariski dense in $X$. Then for any non-empty
    open subset $U \subset X$ and any $A \geq 0$ there exists a choice of anticanonical height function $H_A$
    on $X$ such that
    $$\liminf_{B \to \infty}\frac{N(U,H_A,B)}{B(\log B)^{\rho(X)-1}} > (1 + A) c_{H_A,\textrm{Peyre}}.$$
\end{corollary}

For the varieties covered by Corollary \ref{cor:WA2}, this shows that one cannot remove a closed subset and
hope for Manin's conjecture to hold with Peyre's constant
on the corresponding open subset, uniformly
with respect to every anticanonical height function. In particular, Peyre's conjecture
fails for such varieties. We shall present many examples to which this corollary applies.

Peyre was the first to suggest \cite[\S 8]{Pey03} a ``fix'' to
Manin's conjecture, to make it compatible with the counter-example 
of Batyrev and Tschinkel,
by asking  that one remove  a \emph{thin}, rather than a closed, subset of $X$.
The investigations in  \cite[\S 6]{elsenhans} and 
 \cite{LeR13} are also compatible with this.
Let $X$ be a variety over a number field $k$.
According to Serre  \cite[\S 3.1]{Ser08}, a subset $Z \subset X(k)$ is said to be \emph{thin} if it is
a finite union of subsets which are either contained in a proper closed subvariety of $X$, or contained 
in some $\pi(Y(k))$ where $\pi: Y \to X$ is a generically finite dominant morphism 
of degree exceeding $1$, with $Y$ irreducible.

The examples in our paper are compatible with the following conjecture.

\begin{conjecture} \label{conj:Manin}
	Let $X$ be a Fano variety over a number field $k$ such that $X(k) \neq \emptyset$.
	Then there exists a thin subset $Z \subset X(k)$ such that
	for every anticanonical height function $H$ on $X$ we have
	$$N(X(k) \setminus Z, H,B) \sim c_{H,\textrm{Peyre}}B(\log B)^{\rho(X) - 1}, \quad \text{as } B \to \infty.$$
\end{conjecture}


It is expected that the set of rational points on a Fano variety $X$ defined over $k$
is not thin as soon as it is non-empty. This follows from work of Ekedahl  \cite{Eke90},
together with Colliot-Th\'{e}l\`{e}ne's conjecture that the Brauer--Manin obstruction 
controls  weak approximation for $X$ (see \cite[\S 3.5]{Ser08}). 

We now indicate the contents of this paper. In Section \ref{s:proof} we recall various properties of height functions and establish our main results.  Then in Sections~\ref{s:ci} and \ref{s:fano} we shall give  applications of Theorem \ref{thm:WA1} and Corollary~\ref{cor:WA2} to complete intersections and Fano threefolds, respectively.
In Example~\ref{ex:Fermat}, for instance,  we will
use Theorem~\ref{thm:WA} to
construct examples of hypersurfaces $X$ of arbitrarily large dimension,
over any number field,
where the presence of
saturated  subvarieties precludes
\eqref{con}  holding with $U=X$
and $c_{U,H} = c_{H,\text{Peyre}}$.
Finally, in Section \ref{s:bi} we will apply Corollary \ref{cor:WA2} to
some quadric bundles in biprojective space, where we answer a question of Colliot-Th\'{e}l\`{e}ne.
We explain along the way how these results are compatible with Conjecture \ref{conj:Manin}.

\begin{ack}
While working on this paper the first author
was  supported by ERC grant \texttt{306457}.
The authors would like to thank C\'{e}cile Le Rudulier for her help with formulating
Conjecture \ref{conj:Manin}, together with Lee Butler,  Ulrich Derenthal and the anonymous referee  
for some helpful comments. 
\end{ack}

\section{Proof of the main results}\label{s:proof}

Theorem \ref{thm:WA1} is a simple consequence of the following result.

\begin{theorem} \label{thm:WA}
    Let $X$ be a Fano variety over a number field $k$, let $H$ be an anticanonical
	height function on $X$ and let $Y \subset X$ be a proper closed subvariety of $X$ over $k$.
	Let $Y^{\circ} \subset Y$ be an open subset with $Y^{\circ}(k) \neq \emptyset$.
	Then for any $A \geq 0$ there
    exists a choice of anticanonical height function $H_A$ on $X$ such that
    $$\liminf_{B \to \infty}\frac{N(Y^{\circ},H_A,B)}{N(Y^{\circ},H,B)} > (1 + A) c_{H_A,\textrm{Peyre}}.$$
\end{theorem}

Under the assumptions of Theorem \ref{thm:WA1}, let $Y^\circ = U \cap Y$.
Then Theorem \ref{thm:WA1} is easily deduced from Theorem \ref{thm:WA} on recalling Defintion \ref{d} and noting that
$$
\liminf_{B \to \infty}\frac{N(U,H_A,B)}{B(\log B)^{\rho(X) -1}}
> \liminf_{B \to \infty} \left( \frac{N(Y^{\circ},H_A,B)}{N(Y^{\circ},H,B)} 
\cdot \frac{N(Y^{\circ},H,B)}{B(\log B)^{\rho(X) -1}} \right).
$$

The idea behind the proof of Theorem \ref{thm:WA}
is to build a varying family of height functions $H_\lambda$ for $\lambda > 0$ such that
$H_\lambda(y)=H_1(y)$ for all $y \in Y(k)$, but such that $\lim_{\lambda \to \infty}c_{H_\lambda,\text{Peyre}} = 0.$
Before we carry this plan out, we recall various facts which we shall
need about height functions (see \cite{CLT10} or \cite{Sal98} for proofs).

\subsection{Adelic metrics and height functions}
Let $X$ be a smooth projective variety over a number field $k$ and let $L$ be a line bundle on $X$.
For a place $v$ of $k$, a $v$-adic metric on $L$ is a map which associates to every point $x_v \in X(k_v)$ a function
$\|\cdot\|_v:L(x_v) \mapsto \RR_{\geq 0}$ on the fibre of $L$ above $x_v$ such that
\begin{enumerate}
   		\item For all $\ell \in L(x_v)$, we have $\|\ell\|_v =0$ if and only if $\ell=0$.
   		\item For all $\mu \in k_v$ and $\ell \in L(x_v)$, we have $\|\mu \ell\|_v=|\mu|_v \|\ell\|_v$.
	    \item For any open subset $U \subset X$ and any local section $s \in H^0(U,L)$,
	 the function given by $x_v \mapsto \|s(x_v)\|_v$ is continuous in the $v$-adic topology.
\end{enumerate}
Any model of $(X,L)$ over the ring of integers $\mathfrak{o}_k$ of $k$
naturally gives rise to a $v$-adic
metric on $L$ for each non-archimedean place $v$ of $k$
(see e.g. \cite[\S 	2]{CLT10} or \cite[Def.~2.9]{Sal98}). We then
define an \emph{adelic metric} on $L$ to be  a collection $\|\cdot\|=\{\|\cdot\|_v\}_{v \in \Val(k)}$ of $v$-adic metrics
for each place $v \in \Val(k)$, such that all but finitely many of the $\|\cdot\|_v$ are defined by a single model of $(X,L)$ over $\mathfrak{o}_k$.
We denote the associated adelically metrised line bundle by $\LL=(L,\|\cdot\|)$.
To such an adelic metric one may associate a height function, given by
$$H_{\LL}(x) = \prod_{v \in \Val(k)}\|s(x)\|^{-1}_v,$$
where $s$ is any local section of $L$ defined and non-zero at $x \in X(k)$. This definition is independent of $s$ by
the product formula. We define the class of anticanonical height functions on $X$ to be the class of
height functions associated to a choice of adelic metric on the anticanonical bundle $\omega_X^{-1}$.
If $\LL=(L,\|\cdot\|)$ is an adelically metrised line bundle with $L$ \emph{ample}, then the height function $H_{\LL}$
satisfies the key property that the set
$$\{ x \in X(k): H_{\LL}(x) \leq B\}$$
is finite,
for any $B >0.$
 In this case the counting function
$$
N(Z,H_\LL,B)=\#\{x\in Z: H_\LL(x)\leq B\}
$$
is well-defined 
for any $B > 0$ and any subset $Z\subset X(k)$.
Moreover, if $\LL_1$ and $\LL_2$ are any two adelically metrised line bundles defined on the same
ample line bundle $L$, then the function $H_{\LL_1}/H_{\LL_2}$ on $X(k)$
is bounded (see \cite[Rem.~1.6(b)]{Sal98}). In particular,
$$
\liminf_{B \to \infty}\frac{N(Z,H_{\LL_1},B)}{N(Z,H_{\LL_2},B)}
$$
exists, is finite and non-zero for any subset $Z\subset X(k)$.
We shall regularly use this fact throughout this paper.
It shows that Definition~\ref{d}
is independent of the choice of anticanonical height function, as claimed in the introduction.

\begin{example}
    Consider the ample line bundles
    $\OO_{\PP^n}(d)$ on $\PP^n$, for positive integers $n,d$. Choose a collection of global sections
    $s_0,\dots,s_m \in H^0(\PP^n,\OO_{\PP^n}(d))$ which generate $\OO_{\PP^n}(d)$.
    These may be identified with a collection of homogeneous polynomials of degree $d$ in $(n+1)$ variables,
    whose common zero locus in $\PP^n$ is empty. Then for any place $v$ of $k$, any $x_v \in \PP^n(k_v)$
    and any local section $s$ of $\OO_{\PP^n}(d)$ which is defined
    and non-zero at $x_v$, the expression
    \begin{equation} \label{ex:metric}
    	\|s(x_v)\|_v = \max\left\{ \left|\frac{s_0(x_v)}{s(x_v)}\right|_v,\dots,\left|\frac{s_m(x_v)}{s(x_v)}\right|_v\right\}^{-1},
    \end{equation}
    defines a $v$-adic metric on $\OO_{\PP^n}(d)$. Moreover the collection of these $v$-adic metrics
    defines an adelic metric on $\OO_{\PP^n}(d)$. Using the product formula, one may write the associated height of a point
    $x \in \PP^n(k)$ as
    \begin{equation} \label{ex:height}
    	\prod_{v \in \Val(k)} \max \{ |s_0(x)|_v,\dots,|s_m(x)|_v\}.
    \end{equation}
    Note that on taking $m=n$ and $s_i = x_i^d$ for $i=0,\dots,n$, then the height function (\ref{ex:height}) is simply the $d$-th
    power of the usual height on projective space.
\end{example}

\subsection{Proof of Theorem \ref{thm:WA}}
Let $X$ be a Fano variety over a number field $k$.
By definition, there exists  $a \in \NN$ such that $\omega_X^{-a}$ is very ample. We choose an embedding $X \subset \PP^n$
together with an isomorphism $\omega_X^{-a} \cong \OO_X(1)$.
Next,   let $Y \subset X$ be a proper closed subvariety of $X$ over $k$.
Then we may also choose  $b \in \NN$ together
with a non-zero global section $s_0$ of $H^0(X,\OO_X(b))$ such that
\begin{equation} \label{eqn:Y}
		Y \subset \{s_0 = 0\}.
\end{equation}
Such a global section exists since $Y$ is a \emph{proper} subvariety of $X$
($s_0$ may be identified with a homogeneous polynomial of degree $b$ in $(n+1)$ variables).
Next, choose global sections $s_1,\dots,s_m$ such that
$s_0,\dots,s_m$ generate $\OO_X(b)$ (e.g. they form a basis for $H^0(X,\OO_X(b))$).
We shall now construct a family of metrics on $\omega_X^{-1}$,
defined in a similar manner to the metrics (\ref{ex:metric}).
Fix a place $w$ of $k$. Then for any $\lambda >0$ and any $x_w \in X(k_w)$ we define a $w$-adic metric
on $\omega_X^{-1}$ by
$$\|s(x_w)\|_{w,\lambda} = \max\left\{ \lambda \cdot \left|\frac{s_0(x_w)}{s^{ab}(x_w)}\right|_w,\left|\frac{s_1(x_w)}{s^{ab}(x_w)}\right|_w,
\dots,\left|\frac{s_m(x_w)}{s^{ab}(x_w)}\right|_w\right\}^{-1/ab},$$
for any local section $s$ of $\omega_X^{-1}$ which is defined and non-zero at $x_w$.
Note that here we are identifying $s^{ab}$ with a local section
of $\OO_X(b)$ via 
the choice of isomorphism $\omega_X^{-a} \cong \OO_X(1)$.
It is easy to see that this defines a $w$-adic metric (the exponent $(-1/ab)$ is here to guarantee that
$\|(\mu s)(x_w)\|_{w,\lambda} = |\mu|_w \| s(x_w)\|_{w,\lambda}$ for any $\mu \in k_w$).
For $v \neq w$, we choose 
 fixed $v$-adic metrics $\|\cdot\|_v$ on $\omega_X^{-1}$, almost all of which are defined by a model,
and let $H_\lambda$ denote the height function associated to the corresponding adelic metric. Note that by (\ref{eqn:Y}) we obviously have
\begin{equation} \label{eqn:equal}
	N(Y^{\circ},H_{\lambda},B) = N(Y^{\circ},H_1,B),
\end{equation}
for all open subsets $Y^{\circ} \subset Y$, all $\lambda>0$ and all $B>0$. Peyre's constant (see \cite[D\'{e}f.~2.5, p.~120]{Pey95}) 
takes the form
$$c_{H_\lambda,\text{Peyre}} = \alpha(X)\beta(X) \tau_{\lambda}(\overline{X(k)}).$$
Here $\tau_{\lambda}$ denotes the Tamagawa measure associated to the height function $H_\lambda$
and $\overline{X(k)}$ denotes the closure of $X(k)$ inside $X(\Adele_k)$. Note that $\alpha(X)$ and $\beta(X)$ are independent
of the choice of height function and so do not concern us. Since $\overline{X(k)} \subset X(\Adele_k)$
we have $\tau_{\lambda}(\overline{X(k)}) \leq \tau_{\lambda}(X(\Adele_k))$. Moreover,
if we let $r=\dim X$, then we have
\begin{align*}
\tau_{\lambda}(X(\Adele_k))
=~&|\Delta_k|^{-r / 2}\lim_{s \to 1}( (s-1)^{\rho(X)} L(s,\Pic X_{\kbar}) )\\
&\times
c_w^{-1} \tau_{w,\lambda}(X(k_w))
\prod_{\substack{v \in \Val(F) \\ v \neq w}}c_v^{-1}\tau_{v}(X(k_v)),
\end{align*}
where $\Delta_k$ denotes the discriminant of $k$ and the $c_v$ are a family of ``convergence factors''.
The only part of the above expression which concerns us is $\tau_{w,\lambda}(X(k_w))$. Therefore, by \eqref{eqn:equal}, 
to prove the result it suffices to show that
\begin{equation}	\label{eqn:Tamagawa_goal}
	\tau_{w,\lambda} (X(k_{w})) \to 0 \quad \text{ as } \quad \lambda \to \infty.
\end{equation}
As $X(k_w)$ is compact it may be covered by finitely many $w$-adic
coordinate patches, so in order to prove (\ref{eqn:Tamagawa_goal}) it suffices
to work locally. Let $x_1,\dots,x_r$ be a collection of $w$-adic local coordinates on some
coordinate patch $U \subset X(k_w)$. On $U$ the measure $\tau_{w,\lambda}$ is given by
$\|s\|_{w,\lambda} |\mathrm{d}x_1|_w\dots|\mathrm{d}x_r|_w,$
where
$$s= \frac{\partial}{\partial x_1} \wedge \dots \wedge \frac{\partial}{\partial x_r}.$$
Hence we have
$$
\tau_{w,\lambda}(U) = \int_{x \in U}
\frac{|\mathrm{d}x_1|_w\dots|\mathrm{d}x_r|_w}
{\max\left\{ \lambda\cdot\left|\frac{s_0(x)}{s^{ab}(x)}\right|_{w},\left|\frac{s_1(x)}{s^{ab}(x)}\right|_{w},
\dots,\left|\frac{s_m(x)}{s^{ab}(x)}\right|_{w}\right\}^{1/ab}}.$$
By continuity and the dominated convergence theorem we obtain
\begin{align*}
	\lim_{\lambda \to \infty}\hspace{-0.1cm}
	\tau_{w,\lambda}(U)
	&=\hspace{-0.1cm}
	\int_{x \in U}
	\lim_{\lambda \to \infty} \frac{|\mathrm{d}x_1|_w\dots|\mathrm{d}x_r|_w}
	{\lambda^{1/ab}
	\max\left\{ \left|\frac{s_0(x)}{s^{ab}(x)}\right|_{w}, \frac{1}{\lambda}\cdot \left|\frac{s_1(x)}{s^{ab}(x)}\right|_{w},
	\dots, \frac{1}{\lambda}\cdot\left|\frac{s_m(x)}{s^{ab}(x)}\right|_{w}\right\}^{1/ab}}\\
	&=0,
\end{align*}
since $ab > 0$. This establishes  (\ref{eqn:Tamagawa_goal}) and therefore completes the proof of Theorem \ref{thm:WA}.
\qed

\section{Complete intersections}\label{s:ci}

We now come to some applications of Theorem \ref{thm:WA} to complete intersections in
projective space.
Throughout this section
 $X \subset \PP^n$ is  a Fano complete intersection of $s$  hypersurfaces
of degrees $d_1,\dots,d_s$ over a number field $k$. We shall always assume that $\dim X > 2$.
The condition that $X$ is Fano simply means that
$X$ is non-singular and
 $d < n+1$, where
$$
 d = d_1 + \dots + d_s.
 $$
Under these assumptions we have
$\Pic X \cong \Pic \PP^n \cong \ZZ$ (see  \cite[Ex.~3.1.25]{Lar07})
and
$\omega_X^{-1}=\OO(n+1-d)$ (see \cite[\S II, Ex.~8.4]{Har77}).
In particular, when $X(k)$  is non-empty,  Manin's conjecture predicts the existence of an open subset $U\subset X$
and a constant $c_{U,H}>0$
such that
\begin{equation} \label{conj:Manin_hypersurface}
	N(U,H,B) \sim c_{U,H} B, \quad \text{as } B \to \infty,
\end{equation}
for any anticanonical height function $H$ on $X$.

\subsection{The Hardy--Littlewood circle method}

Let us begin by considering what the circle method has to say about \eqref{conj:Manin_hypersurface} in the
special case  $k=\QQ$.
Let  $\|\cdot\|$ be an arbitrary norm on $\RR^{n+1}.$
In this case one takes the height function
$$H(x) = \|(x_0,\dots,x_n)\|^{n+1-d},$$
on choosing  a representative $x=(x_0:\dots:x_n)$ such that $\x=(x_0,\dots,x_n)$ is a primitive
integer vector.
Birch's work \cite{Bir62}  implies \eqref{conj:Manin_hypersurface} with $U=X$
and $c_{U,H} = c_{H,\text{Peyre}}$
whenever
$d_1=\dots=d_s$ and
$n\geq \sigma+(d-1)2^d$, where $\sigma$ is the affine dimension of Birch's ``singular locus''
and satisfies $\sigma\leq s$
 (note that an easy induction argument shows that
$(d-1)2^d\geq s(s+1)(e-1)2^{e-1}$ when $d=se$).
Birch's result is proved using the circle method.

In Section \ref{s:linear} we will
use Theorem \ref{thm:WA} to
construct examples of complete intersections of arbitrarily large dimension,
over any number field $k$,
where
\eqref{conj:Manin_hypersurface} cannot hold with $U=X$
and $c_{U,H} = c_{H,\text{Peyre}}$, thanks to the presence of
saturated  subvarieties.
A discussion of how the circle method might be expected to account for this behaviour is found in work of
Vaughan and Wooley \cite[Appendix]{vw}, which we summarise here.

Continuing to assume that  $k=\QQ$, suppose that the hypersurfaces that cut out $X$
are defined by forms $F_i\in \ZZ[x_0,\dots,x_n]$ of degree $d_i$, for $1\leq i\leq s$.
In order to establish \eqref{conj:Manin_hypersurface} it suffices to study the counting function
$$
\hat N(B)=\#\{\x\in \ZZ^{n+1}: \mbox{$\|\x\|\leq B$ and $F_i(\x)=0$ for $1\leq i\leq s$}\},
$$
as $B\rightarrow \infty$.  This amounts to counting integral points of bounded height on the universal torsor over $X$ (note that
the affine cone over $X$ in $\mathbb{A}^{n+1}\setminus \{\mathbf{0}\}$ is the unique universal torsor over $X$ up to isomorphism).
By orthogonality one has
$$
\hat N(B)=\int_{(0,1]^s} S(\bal;\mathbf{F}) \d \bal,
$$
where
$$
S(\bal;\mathbf{F}) =\sum_{\substack{\x\in \ZZ^{n+1}\\ \|\x\|\leq B  }}
e^{2\pi i\{\alpha_1F_1(\x)+\dots+\alpha_s F_s(\x)\}}.
$$
The circle method is concerned with an asymptotic analysis of this integral, based on a suitable dissection of the range of integration into {\em major arcs} and {\em minor arcs}.
The  general philosophy behind the circle method is that the contribution from the
major arcs should lead to a main term in agreement with the right hand side of \eqref{conj:Manin_hypersurface}
with $c_{U,H}=c_{H,\text{Peyre}}$,
while the {\em minor arcs}
should account for the contribution from any accumulating subvarieties.
In fact, according to \cite[Appendix]{vw},
the major arcs should account for the contribution to $\hat N(B)$
from all points not lying
on a proper subvariety of ``integral degree'' at most $d$ and height at most $B^\delta$, for a sufficiently small positive number $\delta$. Here, the integral degree is defined to be the minimum of the sum of degrees, taken over all possible sets of forms defining the subvariety. According to this theory, therefore,  the contribution from any saturated
subvarieties would come from the minor arcs.

\subsection{Linear subvarieties}\label{s:linear}

We now return to the case of a general number field $k$.
Let $r \in \NN$. By \cite{Sch79},
for any $r$-plane $L \subset \PP^n$ defined over $k$
there is a constant $c_{H_0,L} >0$ such that
\begin{equation} \label{eqn:L}
	N(L,H_0,B) \sim c_{H_0,L}B^{r+1},
\end{equation}
where $H_0$ is the height function
$$H_0(x) = \prod_{v \in \Val(k)} \max\{|x_0|_v,\dots,|x_n|_v\},$$
on $\PP^n$.
Note that $H_0^{n+1-d}$ is an anticanonical height on $X$.
As the asymptotic formula \eqref{eqn:L} also holds on every non-empty open subset of $L$,
it follows that any $(n-d)$-plane $L$ over $k$ properly contained in $X$ is saturated.

\begin{theorem}\label{thm:CI}
	Let $X \subset \PP^n$ be a Fano complete intersection of $s$  hypersurfaces
	of degrees $d_1,\dots,d_s$ over a number field $k$. Assume that $\dim X > 2$, that $d = d_1 + \dots + d_s > 1$
	and $d<n$. Suppose that $X$ contains a $(n-d)$-plane $L$ defined over $k$.
	Then for any open subset $U \subset X$ which meets  $L$ and any $A \geq 0$ there 
    exists a choice of anticanonical height function $H_A$ on $X$  such that
	$$\liminf_{B \to \infty}\frac{N(U,H_A,B)}{B} > (1 + A) c_{H_A,\textrm{Peyre}}.$$
\end{theorem}
\begin{proof}
	The result follows on applying Theorem \ref{thm:WA1} to $L$.
	\end{proof}

In particular Theorem \ref{thm:CI} implies that for such complete intersections, the asymptotic formula
 \eqref{conj:Manin_hypersurface} cannot hold with $U=X$
and $c_{U,H} = c_{H,\text{Peyre}}$,
for every choice of anticanonical height
function $H$.

Note that one can make the height functions appearing in Theorem \ref{thm:CI} very explicit.
Assume for simplicity that $X$ is as in Theorem \ref{thm:CI} with $k=\QQ$ and that $X$ contains a $(n-d)$-plane $L$
defined over $\QQ$. Then after a linear change of variables (which will only change any height
functions by a bounded amount), we may assume that
$$L \subset \{x_0 = 0\} \quad \text{ and } \quad X \not \subset \{x_0 = 0\}.$$
For any $\lambda > 0$ we  take the height function
$$H_\lambda(x) = \max\{\lambda|x_0|,|x_1|,\dots,|x_n|\}^{n+1-d},$$
where we have chosen a representative $x=(x_0:\dots:x_n)$ such that $(x_0,\dots,x_n)$ is a primitive
integer vector. Clearly  $H_\lambda(x) = H_1(x)$ for any $x \in L(\QQ)$ and the proof of Theorem \ref{thm:WA}
shows that
$$\lim_{\lambda \to \infty} c_{H_\lambda,\textrm{Peyre}} = 0.$$
This can also be proved by studying the  associated ``singular integral'' in the circle method.

We now give some examples to which Theorem \ref{thm:CI} applies. We first show that
there are examples of arbitrarily large dimension over every number field.

\begin{example}[Fermat hypersurfaces] \label{ex:Fermat}
Let $d\geq 3.$
	For any $n\in \NN$ such that $n >d \geq n/2+1$ and any number field $k$,
	we claim that there is a non-singular hypersurface $X\subset \PP^n$ of degree $d$
	which contains a $(n-d)$-plane defined over $k$.
	For any $r<n$ consider the Fermat hypersurface
	$$X_r: \quad x_0^{d} + \dots + x_{r}^{d} =x_{r+1}^{d} + \dots + x_n^d \quad \subset \PP^n,
	$$
	over $k$. Suppose first that $n=2r+1$ is odd for $r>1$.
	Then $X_r$ contains the $r$-plane
	$x_i = x_{i +r+1}$, for  $i = 0,\dots,r.$
	In particular, $X_r$ contains a $(n-d)$-plane if $d \geq (n+1)/2$ and hence
	Theorem \ref{thm:CI} applies in this case.
	If instead $n=2r$ is even for $r>1$, then $X_r$ contains the
$(r - 1)$-plane given by the equations
	$x_0 = 0$ and $x_i = x_{i + r}$  for $i = 1,\dots,r.$ It follows that  $X_r$ contains a $(n-d)$-plane if
$d \geq n/2 +1$ and hence  	Theorem~\ref{thm:CI} again applies in this case.
\end{example}
Examples of the above type were also considered by Hooley \cite{Hoo86} in the case
$n=5$ and $d=3$ over $\QQ$.
One may construct more general examples to which Theorem \ref{thm:CI} applies using known facts about the Fano
variety of $r$-planes. Recall that the collection of $r$-planes inside $X$ forms a closed subscheme
of the Grassmannian $\mathbb{G}(r,n)$ of $r$-planes in $\PP^n$, which we denote by $F_r(X)$ and call the  \emph{Fano
variety} of $r$-planes inside $X$ (one should not confuse these with the Fano varieties as defined above).
We have the following result.

\begin{lemma}\label{lem:Fano}
	Let $X \subset \PP^n$ be a Fano complete intersection of $s$  hypersurfaces
	of degrees $d_1,\dots,d_s$ over a field $k$ of characteristic zero. 
	Then $F_r(X)$ is non-empty (as a scheme) if
	\begin{equation} \label{eqn:F_r_non_empty}
        \min\left\{(r+1)(n-r) - \sum_{i=1}^s{d_i+r \choose r}, n -2r -s\right\} \geq 0.
    \end{equation}
    Moreover if \eqref{eqn:F_r_non_empty} does not hold and $X$ is general, then $F_r(X)$ is empty.
	If \eqref{eqn:F_r_non_empty} holds then
	\begin{equation}\label{eq:lower}\dim F_r(X) \geq (r+1)(n-r) - \sum_{i=1}^s{d_i+r \choose r}.\end{equation}
	If \eqref{eqn:F_r_non_empty} holds and $X$ is general, then one has equality in \eqref{eq:lower} and
	$F_r(X)$ is smooth and geometrically integral.
\end{lemma}
\begin{proof}
This follows from work of  Debarre and Manivel
 \cite[Thm.~2.1]{DM98}. The only thing that requires further explanation is that
 one always has the lower bound \eqref{eq:lower}
when  \eqref{eqn:F_r_non_empty} holds.
Complete intersections $X=X_{\mathbf{F}}$ as in the statement of the lemma are defined by
a system $\mathbf{F}=(F_1,\dots,F_s)$ of forms
in $\mathrm{Sym}^{\mathbf{d}}~V^*=\bigoplus_{i=1}^s \mathrm{Sym}^{d_i}~V^*$, where $V$ is a $k$-vector space of dimension $n+1$.
Consider the incidence correspondence
$$
I_r=\{(\mathbf{F},\Lambda)\in \mathrm{Sym}^{\mathbf{d}}~V^* \times \mathbb{G}(r,n): \Lambda\subset X_{\mathbf{F}}\}.
$$
Then the proof of
\cite[Thm.~2.1]{DM98} shows that $I_r$ is
geometrically irreducible of codimension
$\sum_{i=1}^s{d_i+r \choose r}$
in $\mathrm{Sym}^{\mathbf{d}}~V^* \times \mathbb{G}(r,n)$.
Let $Y$ be the image of $I_r$ under the projection $\pi_1:I_r\rightarrow \mathrm{Sym}^{\mathbf{d}}~V^*$.
We note that the fibre over any $\mathbf{G}\in Y$ is the Fano variety $F_r(X_{\mathbf{G}})$.
It follows from \cite[Cor.~11.13]{harris} that
\begin{align*}
\dim I_r
&=\dim Y+ \inf_{\mathbf{G}\in Y} \dim \pi_1^{-1}(\mathbf{G})\\
&\leq  \dim
\mathrm{Sym}^{\mathbf{d}}~V^*
+ \dim F_r(X),
\end{align*}
for any Fano complete intersection $X\subset \PP^n$ as in the statement of the lemma.
This establishes  \eqref{eq:lower}, on recalling that $\dim \mathbb{G}(r,n)=(r+1)(n-r).$
\end{proof}

We apply this result to study complete intersections with  $d=n-1$, as follows.

\begin{example} \label{ex:n-1}
Let  $s$ be arbitrary, with  $d=n-1$. In this case we are  interested in lines.
A simple calculation using Lemma \ref{lem:Fano} shows that $X$ obtains a line over a finite
field extension $K$ of $k$ and so Theorem \ref{thm:CI} applies to $X_K$. Note that for a
\emph{general} such complete intersection we have $\dim F_1(X) = n-s-1$ by Lemma~\ref{lem:Fano},
and any point on $X$ lies on a line defined over $\kbar$ (see \cite[Rem.~3.3]{zav}).
Now let  $s=1$ and let $X\subset \PP^n$ be a non-singular hypersurface of degree $d=n-1$.
It has been conjectured by Debarre and de Jong that  we always have $\dim F_1(X) = n-2$ and Beheshti \cite{zav} has proved this for $d\leq 6$.
It would be interesting to find an example  for which the
rational points on the Fano variety of lines are Zariski dense. Note that Beheshti \cite[Thm.~2.1]{zav}
has  shown that  $F_1(X)$
contains no uniruled divisors, whence $F_1(X)$ itself cannot be uniruled.
\end{example}

We can say more in some special cases.
Elsenhans and Jahnel \cite{EJ06} have produced computational evidence  for
diagonal cubic threefolds over $\QQ$, which
suggests that one needs to remove all the $\QQ$-lines if one hopes to
obtain Peyre's constant in Manin's conjecture.
Using Lemma~\ref{lem:Fano} and Theorem~\ref{thm:CI}, we can show  that this is indeed the case.

\begin{example}[Cubic threefolds]\label{Ex:cubic}
This example is based on the investigation of Bombieri and Gubler \cite[\S 11.10]{bomb}.
Suppose that $s=1,n=4$ and $d=3$, i.e.\ $X$ is a non-singular cubic threefold.
Example \ref{ex:n-1} shows that  $X$ contains a line over a finite field extension $K$ of $k$ and so
Theorem \ref{thm:CI} applies to $X_K=X\times_k K$. It is classically known 
that the Fano variety $F_1(X)$ of lines is a surface and that the lines geometrically sweep out all
of $X$ and so, over $\kbar$, every point on $X$ lies on a line. However, the $K$-lines
cannot be Zariski dense over a number field and so Corollary \ref{cor:WA2} does not apply.
Indeed, $F_1(X)$ is a surface of general type which embeds inside its Albanese variety and hence
the rational points on $F_1(X)$ are not dense over any number field by a theorem of Faltings \cite{Fal91}.
\end{example}

Recall that for a variety $V$ defined over $k$, the rational points of $V$ are said to be {\em potentially dense}
if there exists a finite extension $K$ of $k$ such that $V(K)$ is Zariski dense in $V$ (see \cite{Has03}).
The previous example shows that potential density fails for the Fano variety of lines in a cubic threefold.
We now come to an example to which Corollary \ref{cor:WA2} applies.

\begin{example}[Intersections of two quadrics]\label{Ex:2_quadrics}
This example has also been considered by Hassett, Tanimoto and Tschinkel \cite[Example~27]{sho}.
Suppose that $s=2, n=5$ and $d_1=d_2=2$, i.e.\ $X$ is a non-singular intersection of two quadrics
in $\PP^5$.  Example \ref{ex:n-1} shows that  Theorem \ref{thm:CI} applies to $X$ over some finite
field extension of $k$.
In fact, one may deduce something stronger.

The Fano variety of lines $F_1(X)$ is classically known to be a principal homogeneous
space of the Jacobian of an explicit curve of genus two (see \cite[Thm.~17.0.1]{CF96}),
and moreover the lines sweep out all of $X$.
It is well-known that potential density holds for abelian varieties (see \cite[Prop.~4.2]{Has03}).
In particular there exists a finite field extension
$K$ of $k$ such that the collection of lines defined
over $K$ is Zariski dense in $X_K$. Therefore Corollary \ref{cor:WA2} applies. We deduce that there
is no choice of open subset $U$ of $X_K$ for which Manin's conjecture
holds with Peyre's constant, uniformly with respect to every choice
of anticanonical height function.

For completeness, we now construct such an  example over $\QQ$.
The following computations were performed using \texttt{Magma} \cite{Magma}.
Consider the following intersection $X$ of two quadrics
\begin{align*}
	x_0^2 + x_1^2 &= x_2^2 +  x_3^2 + x_4^2 + x_5^2 + 8x_0x_5 \\
	2x_0x_1 &= 2x_2x_3 + x_5^2,
\end{align*}
viewed over $\QQ$.
One easily checks that $X$ is non-singular.
The corresponding hyperelliptic curve is given by
$$C: \quad y^2 = x^5 - x^4 - 2x^3 + 18x^2 + x - 17.$$
Denote by $J = \Jac(C)$ the Jacobian of $C$.
Note that $C$ has a rational point at infinity, which corresponds to the fact that the pencil
of quadrics given by $X$ contains a degenerate quadric over $\QQ$.
The line $x_0 = x_2, x_1=x_3, x_4=x_5=0$ lies in $X$, in particular there is an isomorphism
$F_1(X) \cong J$ over $\QQ$. The rank of the Mordell-Weil group of $J$ is $1$. Hence to prove that
the rational points on $J$ are Zariski dense, it suffices to show that $J$ is simple as an abelian
variety (see the proof of \cite[Prop.~4.2]{Has03}). Now, $J$ has good reduction at $3$ and moreover
the characteristic polynomial of the  Frobenius endomorphism on the Tate module at $3$ is irreducible,
so by a theorem of Tate \cite[Thm.~1(b)]{Tat66}
we see that the reduction of $J$ modulo $3$ is simple. Therefore $J$ itself must be simple.
Hence the collection of $\QQ$-lines on $X$ is Zariski dense over $\QQ$, and moreover this obviously also holds
over every finite field extension of $\QQ$.
Other examples may be constructed in a similar manner.
\end{example}

Even though the collection of rational points which lie on a $k$-line in $X$
can be Zariski dense in the previous example, we shall nonetheless show in \S \ref{sec:compatible}
that the collection of such points is  \emph{thin} in $X$. In particular, this example
is  compatible with Conjecture \ref{conj:Manin}.
The following  example involves planes instead of lines.

\begin{example}[Intersections of $n/2-1$ quadrics]\label{ex:n/2-1}
Suppose now that $n>4$ is even and that $X\subset \PP^n$ is a non-singular
complete intersection of $s = n/2-1$  quadrics.
Then  $d= n-2$ and we are interested in planes. A simple calculation using
Lemma \ref{lem:Fano} shows that $F_2(X)$ is non-empty, hence $X$ obtains a plane over some finite
field extension $K$ of $k$ and Theorem \ref{thm:CI} applies to $X_K$. Note that for a general such $X$
we have $\dim F_2(X)=0$.  Thus for general $X$ the  planes sweep out a proper  Zariski closed subset.
\end{example}

Using  Lemma \ref{lem:Fano}, we may now show that Examples \ref{ex:n-1} and \ref{ex:n/2-1}  essentially exhaust
all possible applications of Theorem \ref{thm:CI} for \emph{general} complete intersections.
To this end it suffices to consider the case where $d_i \geq 2$ for each $i=1,\dots,s$.

\begin{theorem} \label{thm:n-d_classification}
	Let $X \subset \PP^n$ be a Fano complete intersection of $s$ hypersurfaces
	of degrees $d_1,\dots,d_s\geq 2$ over a field $k$ of characteristic zero. Assume that  $d=d_1 + \dots + d_s<  n-1$  and that if $n$ is even then $X$ is
    not an intersection of $n/2-1$ quadrics. If $X$ is general, then it does not contain a $(n-d)$-plane.	
\end{theorem}
\begin{proof}
    To prove the result we need  to show that $F_{n-d}(X)$ is empty, for general $X$.
	Lemma \ref{lem:Fano}  leads to purely combinatorial problem. Namely, we need to show that if
	$r = n - d$ and $2 \leq r \leq (n-s)/2$
    then
    $$\sum_{i=1}^s{d_i+r \choose r} \geq (d_1 + \dots + d_s)(r+1),$$
	with equality if and only if $(d_1,\dots,d_s) = (2,\dots,2)$ and $r=2$.
	It will suffice to show that for any integers  $d,r \geq 2$ we have
	$$
		{d + r \choose r} \geq d(r +1),
	$$
with equality if and only if $d=r=2$.
But ${d + r \choose r}$ is just the number of monomials of degree $d$ in $r+1$ variables.
Among these are  the monomials $x_i^d,x_i^{d-1}x_{i+1},\dots,x_ix_{i+1}^{d-1}$, for $0\leq i\leq r$, where the indices are taken modulo $r$.
This establishes the desired lower bound since these monomials
 number
$d(r+1)$ in total. It is clear that the latter set of monomials exhaust all monomials of degree $d$ in $r+1$ variables precisely when $d=r=2.$
\end{proof}

Example \ref{ex:Fermat} shows that  there are complete intersections not of the form given in Example \ref{ex:n-1}
or Example \ref{ex:n/2-1} to which Theorem \ref{thm:CI} applies.
This does not contradict Theorem \ref{thm:n-d_classification}
since it only applies to
general complete intersections.

So far we have focused on non-singular complete intersections $X\subset \PP^n$ which contain $(n-d)$-planes,
with  $d=d_1+\dots+d_s<n$. Suppose now that  $d=n.$
In this setting  $(n-d)$-planes are obviously not saturated. However, Lemma~\ref{lem:Fano} implies that
such  $X_{\kbar}$ always contain a $(n+1 - d)$-plane, i.e.\ a line.
For general $X$ one  has $\dim F_1(X) = n -s - 2$,
and hence the collection of lines sweep out a proper closed subset in $X$.
In this setting any $k$-line  is expected to be strongly accumulating for any choice of anticanonical height $H$ on $X$.

\subsection{Compatibility with Conjecture \ref{conj:Manin}} \label{sec:compatible}
We now explain how the results of this section are compatible with Conjecture \ref{conj:Manin}. For simplicity
we focus on the case of lines on a general complete intersection of the type considered in Example \ref{ex:n-1}.

\begin{lemma} \label{lem:compatible}
	Let $X \subset \PP^n$ be a general non-singular complete intersection
	of $s$ hypersurfaces of degrees $d_1,\ldots,d_s$ over a number field
	$k$ with $d = n-1$. Then the collection of rational points on $X$ which lie on a $k$-line in $X$
	is thin.
\end{lemma}
\begin{proof}
	Write $F = F_1(X)$. As $X$ is general,  by Lemma~\ref{lem:Fano} we know that
	$F$ is smooth and geometrically integral, with $\dim F = n - s -1$.
	The universal family of lines in $X$ is a variety $V \subset F \times X$.
	The fibres of the projection $\pi_1$ onto $F$ correspond to the lines inside $X$, with the projection
	$\pi_2$ onto $X$ corresponding to the inclusion of the lines into $X$. In particular $\pi_2(V(k))$
	consists of exactly those rational points on $X$ which lie on a $k$-line contained in $X$.
	Since $\dim F = n - s - 1$,  we have $\dim V = n - s$. As explained in Example~\ref{ex:n-1},
	the map $\pi_2$ is dominant.
	To show that $\pi_2(V(k))$ is thin, it suffices to show that $\pi_2$ is not a birational morphism.
	By \cite[Thm.~2.1]{zav} we know that $F$ is not uniruled.
	In particular $F$  is not rationally connected and so
	 $V$ is also not rationally connected. But $X$ is
	 a Fano variety and so is
	rationally connected (see \cite[Cor.~6.2.11]{IP98}), whence it
	 cannot be birational to $V$.
\end{proof}

The key result \cite[Thm.~2.1]{zav} used in the above proof applies to arbitrary non-singular
complete intersections  $X$ of the above type (not just general $X$). Assuming the truth
of the conjecture of Debarre and de Jong mentioned in Example \ref{ex:n-1},
it therefore seems likely that the conclusion of Lemma \ref{lem:compatible}
should hold for arbitrary non-singular hypersurfaces with $d = n-1$.
It is also possible to prove the analogue of Lemma \ref{lem:compatible}
for \emph{any} non-singular intersection of two quadrics $X$ over $k$, as in Example \ref{Ex:2_quadrics}.
In this case  $F_1(X_{\kbar})$ is an abelian surface, which is certainly
not rationally connected. The same proof given in Lemma~\ref{lem:compatible} therefore
shows that the collection of rational points which lie on a $k$-line in $X$ is thin.

\section{Fano threefolds}\label{s:fano}
We now come to some applications of Theorem \ref{thm:WA1} to Fano threefolds.

\begin{definition}\label{def:index}
The {\em index}
of a Fano variety $X$ over a  field $k$, denoted by
 $\Index(X)$,
  is the largest positive integer
$r$ such that $-K_X = r \cdot D$ for some $D \in \Pic X$.
\end{definition}

Note that in general this definition depends on the choice of ground field (e.g.\ a conic without a rational point
has index one, whereas $\PP^1$ has index two).
Let $X$ be a Fano variety defined over a number field $k$, with geometric Picard group
 $\Pic X_{\kbar}$. All of the Fano varieties $X$ appearing in this section
will be assumed to contain a $k$-point, so that
$\Pic X\cong (\Pic X_{\kbar})^{\mathrm{Gal(\kbar/k)}}$ and hence $\Index(X) = \Index(X_{\kbar})$.
We shall focus our attention on
 Fano threefolds of (geometric) Picard number one, by which we mean that
 $\Pic X_{\kbar}\cong \ZZ$.
Since $X(k)\neq \emptyset$, we have
$\Pic X\cong \Pic X_{\kbar}$.
The Fano varieties considered in Examples~\ref{ex:n-1} and \ref{ex:n/2-1} have index $2$ and $3$,
respectively.
As a general rule, the arithmetic of Fano varieties of Picard number one  seems harder to understand as the  index gets smaller.

There has been an industrious  research programme  on Manin's conjecture for
del Pezzo surfaces, the Fano varieties of dimension two.
In stark  contrast to this, very little is known about the Manin conjecture for Fano threefolds $X$,
aside from some special threefolds, such as those equipped with a group action
with an open dense orbit (e.g.\ toric varieties). The only general result is by Manin \cite[Thm.~0.4]{manin},
who has shown that after a possible extension of the ground field,
for all non-empty open subsets $U\subset X$ we have 
$$
 \limsup_{B\rightarrow \infty} \frac{\log N(U,H,B)}{\log B}\geq 1,
$$
for any anticanonical height $H$ on $X$.

Let $X$ be a Fano threefold of Picard number one.
In what follows we fix an ample generator $D_X$ for $\Pic X$, which gives the
isomorphism $\Pic X\cong \ZZ$.
In particular,  we have $-K_X = \Index(X)\cdot D_X$.
We define the {\em degree} of $X$ to be the integer
$
d(X)=D_X^3$
and
the \emph{genus} of $X$ to be $g(X)=(-K_X)^3/2 + 1$,
which is also an integer.
The Fano threefolds of Picard number one
have been classified by Iskovskikh (see \cite{IP98}).
Using this classification we shall show that,
except in a few special cases, saturated subvarieties occur and, moreover, 
the union of the saturated
subvarieties can even be Zariski dense over a number field.
We also give some examples for which the union of saturated subvarieties
under consideration can never be Zariski dense over a number field, even
though they are dense over the algebraic closure. Our result is as follows.

\begin{theorem}\label{thm:Fano3}
	Let $X$ be a Fano threefold
	over a number field $k$
	with Picard number one  such that $X(k) \neq \emptyset$.
	Assume that $D_X$ is very ample.
	If $\Index(X) = 1$ or $2$ then there exists a finite field extension $k \subset K$
	such that $X_K$ contains a saturated subvariety $Y$.
	For any open subset $U \subset X_K$ which meets $Y$ and any $A \geq 0$ there
    exists a choice of anticanonical height function $H_A$ on $X_K$  such that
	$$\liminf_{B \to \infty}\frac{N(U,H_A,B)}{B} > (1 + A) c_{H_A,\textrm{Peyre}}.$$
	Moreover, suppose that $(\Index(X),d(X)) \in \{(2,4),(2,5) \}$.
	Then there exists a finite field extension $k \subset K$ such that
	for any non-empty open subset $U \subset X_K$ and any $A \geq 0$ there
    exists a choice of anticanonical height function $H_A$ on $X_K$ such that
	$$\liminf_{B \to \infty}\frac{N(U,H_A,B)}{B} > (1 + A) c_{H_A,\textrm{Peyre}}.$$
\end{theorem}
In particular for the latter Fano threefolds, there is no choice of
open subset $U \subset X$ which is uniform with respect to every height function,
for which Manin's conjecture holds with Peyre's constant.
We shall show in Lemma \ref{lem:3folds_thin}, however, that these results are
still compatible with Conjecture \ref{conj:Manin}.

\begin{remark}
	Let $X$ be a Fano threefold over $\bar{k}$ with the property that $(\Index(X),d(X))=(2,5)$. Such threefolds are 
	equivariant compactifications of a homogeneous space for $\SL_2$ \cite[Thm.~3.4.8]{IP98}.
	Namely, there exists an action of $\SL_2$ on each such Fano variety with an open dense orbit,
	such that the stabiliser of a general point is the binary octahedral subgroup of $\SL_2$.
	
	Manin's conjecture has been proved for many equivariant compactifications of homogeneous spaces
	using the methods of harmonic analysis. In all cases where this approach has been successful however, 
	the stabiliser of a general point has been geometrically connected. Given the presence of a Zariski 
	dense set of saturated subvarieties here, it is not clear that 
	the harmonic analysis approach could be made to apply.
\end{remark}

\subsection{Rational curves on Fano varieties}
For the proof of Theorem \ref{thm:Fano3}
we shall need the following result on the arithmetic of rational curves on Fano varieties.

\begin{lemma} \label{lem:curves}
    Let $X$ be a Fano variety over a number field $k$ and let $C \subset X$ be a non-singular
    rational curve on $X$ defined over $k$ such that $C(k) \neq \emptyset$.
    Then for any choice of anticanonical height function $H$ on $X$ we have
    $$N(C,H,B) \sim c_{C,H}B^{2/r},$$
    where $c_{C,H}>0$ and $r = C \cdot (-K_X) > 0$.
\end{lemma}
\begin{proof}
    As $C(k) \neq \emptyset$ we have $C \cong \PP^1$.
    Moreover by assumption there is an isomorphism $\omega_X^{-1}|_C \cong \OO_{\PP^1}(r)$.
	We are therefore reduced to counting points of bounded height on $\PP^1$,
	for which the desired conclusion follows from \cite[Cor.~6.2.18]{Pey95}.
\end{proof}

Non-singular rational $k$-curves $C$ on such $X$ with $C\cdot(-K_X)=1$ or $2$
are of particular interest to us, as in the first case they are expected to be strongly
accumulating and in the second case they are saturated if $X$ has Picard
number one.
In view of Definition \ref{def:index} we note that
 $C \cdot (-K_X) \geq \Index(X)$ for all curves $C$ on $X$.
In particular if $\Index(X) > 2$, then curves of the above type cannot occur
(which explains why we only consider the case $\Index(X) \leq 2$ in Theorem~\ref{thm:Fano3}).

\subsection{Proof of Theorem \ref{thm:Fano3}}

We now consider Fano threefolds $X$ of Picard number one, paying particular attention to
the saturated subvarieties, following Iskovskikh's classification (see \cite{IP98}).
The classification of Iskovskikh implies that $1 \leq \Index(X) \leq 4$.
We define a line in $X$ to be a curve $L$ of arithmetic genus zero
such that $L \cdot D_X =1$,
and a conic in $X$ to be a curve $C$ of arithmetic genus zero such that $C \cdot D_X=2$.
There is a simple characterisation of such curves if $D_X$ is \emph{very} ample
(we assume this condition from now on).
Namely consider $X \subset \PP^n$ given by an embedding induced by $D_X$
and let $C \subset X$ be a curve on $X$.
If $C \cdot D_X=1$ then $C$ is a line in $\PP^n$ and if
$C \cdot D_X=2$ then $C$ is a conic in $\PP^n$ (see \cite[Ex.~I.7.8]{Har77}).
In particular we see that for such curves, the assumption that they
have arithmetic genus zero is redundant.

\subsection*{$\Index(X) = 4$}
In this case $X \cong \PP^3$. Here
Manin's conjecture with the usual height function is known by Schanuel's theorem
(see \cite[Cor.~6.2.18]{Pey95} for general height functions).

\subsection*{$\Index(X) = 3$}
Here $X$ is a non-singular quadratic hypersurface in $\PP^4$, which is an
example of a flag variety. Manin's conjecture is therefore known by \cite{FMT89}.

\subsection*{$\Index(X) = 2$}
This is the first interesting case. Since $-K_X=2D_X$, there are
no curves $C$ on $X$ such that $C\cdot (-K_X) = 1$. One has $1 \leq d(X) \leq 5$.
In fact we have $d(X) \geq 3$ since $D_X$ is assumed to be very ample.
Lemma \ref{lem:curves} implies that the $k$-lines on $X$ are saturated.
The following cases arise:

\begin{itemize}
	\item $d(X)=3$. These are cubic threefolds in $\PP^4$, as  in Example~\ref{Ex:cubic}.
	\item $d(X)=4$. These are complete intersections of two quadrics in $\PP^5$,
	as in Example \ref{Ex:2_quadrics}.
	\item $d(X)=5$. This is the first new case. Here $X$ is the intersection of a linear subspace of codimension three
	with the Grassmannian $\GG(1,4)$ given by the Pl\"{u}cker embedding in $\PP^9$.
	There is an isomorphism $F_1(X_{\kbar}) \cong \PP^2$ and the lines sweep out all of $X$ (see \cite[Prop.~3.5.6]{IP98}).
	In particular the $k$-lines are Zariski dense in $X$
	as soon as there is a single $k$-line on $X$, and so Corollary \ref{cor:WA2} applies.
\end{itemize}

\subsection*{$\Index(X) = 1$}
Here the fundamental invariant
is the genus $g(X)$, as defined above.
One has $2 \leq g(X) \leq 12$ and $g(X) \neq 11$.
Since $-K_X$ is assumed to be very ample we have
 $g(X)\geq 3$.

By Lemma \ref{lem:curves}, any $k$-line in $X$ is expected to be strongly accumulating.
It is known that $F_1(X)$ is non-empty as a scheme (see \cite[Thm.~4.6.7]{IP98})
and has dimension one (see \cite[Prop.~4.2.2]{IP98}). Thus
the lines sweep out a proper closed subvariety of $X$. Moreover for general $X$,
it is known (see \cite[Thm.~4.2.7]{IP98}) that $F_1(X)$ is a non-singular integral projective
curve whose genus is at least $3$. Hence Faltings's Theorem \cite{faltings} tells us that there
can only be finitely many lines defined over $k$.

Next let $C(X)$ denote the Hilbert scheme of conics on $X$.
By Lemma \ref{lem:curves}, we know that any $k$-conic in $X$ with a rational point is saturated.
The Hilbert scheme of conics on $X$ is non-empty (see \cite[Thm.~4.6.7]{IP98}), 
has dimension $2$ (see \cite[Prop.~4.2.5]{IP98}) 
and every point over $\kbar$ lies on a conic.
The current state of knowledge for $C(X)$ is summarised in work of  Iliev and Manivel \cite[\S 3]{IM07}. For general $X$ we have the following behaviour:

\begin{itemize}
    \item $g(X)=12$. In this case $C(X_{\kbar}) \cong \PP^2$ (in fact this is true for any
    such $X$, not just general $X$, see \cite[Lem.~5]{Fae14}). In particular,
    as soon as $X$ contains a $K$-conic, for a finite field extension $k\subset K$,
    the collection of $K$-conics is Zariski dense.
    \item $g(X)=10$. Here $C(X_{\kbar})$ is an abelian surface,
    so again the $K$-conics are Zariski dense after a finite field extension $k \subset K$.
    \item $g(X)=9$. Here $C(X_{\kbar})$ is a ruled surface over a smooth plane quartic curve.
    Hence Faltings's theorem \cite{faltings} tells us that the rational points on $C(X)$ can never
    be Zariski dense.
    \item $g(X)=8$. Here $C(X_{\kbar})$ is isomorphic to the Fano variety of lines inside a cubic
    threefold. We have already seen in Example \ref{Ex:cubic} that the rational points cannot be Zariski dense here.
    \item $g(X)=7$. Here $C(X_{\kbar}) \cong V^{(2)}$ is the symmetric square of a curve $V$ of genus $7$. The Albanese variety of $V^{(2)}$ is simply the Jacobian $\Jac(V)$ of $V$. Moreover the Albanese map
	$\alpha: V^{(2)} \to \Jac(V)$ is an embedding if $V$ is not hyperelliptic, and contracts a single rational curve $\Gamma$ if $V$ is hyperelliptic
	(see \cite[\S 2.4]{LP12}).
	In either	case the image of $\alpha$ is not a translate of an abelian subvariety, hence the rational points on $C(X)$
	can never be Zariski dense by Faltings's theorem \cite{Fal91}. Moreover if $V$ is neither hyperelliptic nor bielliptic, then $C(X)$ even has only finitely many rational points (see the proof of \cite[Cor.~3]{HS91}).
    \item $3 \leq g(X) \leq 6$. Examples here include quartics in $\PP^4$,
	complete intersections of a quadric and a cubic in $\PP^5$
    and complete intersections of three quadrics in $\PP^6$.
    It is known that $C(X)$ has general type, so the Bombieri--Lang conjecture
    predicts that the rational points on $C(X)$ are never Zariski dense.
    It would be interesting if one could prove that this is indeed the case.
\end{itemize}

This completes the proof of Theorem \ref{thm:Fano3}. \qed

We finish our analysis
by explaining how Theorem \ref{thm:Fano3} is compatible with Conjecture \ref{conj:Manin}.
For simplicity we only consider the case of general Fano threefolds of Picard number one.

\begin{lemma} \label{lem:3folds_thin}
	Let $X$ be as in Theorem \ref{thm:Fano3} and assume that $X$ is general.
	If $X$ has index $2$ (resp.\ $1$) then the collection of rational
	points in $X$ which lie on a $k$-line (resp.\ $k$-conic) is thin.
\end{lemma}
\begin{proof}
	The proof is similar to the proof of Lemma \ref{lem:compatible}.
	Denote by $F$ the Hilbert scheme of lines (resp.\ conics) on $X$ if $X$ has index $2$ (resp.\ $1$).
	Denote by $V \subset F \times X$ the corresponding universal family. If
	$(\Index(X),d(X))\neq(2,5)$ and $(\Index(X),g(X))\neq(1,12)$, then we have seen that $F$
	is not rationally connected. The same argument given in the proof of Lemma \ref{lem:compatible} shows that
	$\pi_2(V(k))$ is thin inside $X(k)$, where $\pi_2$ denotes the second projection.
	To handle the other cases, we note that the degree of $\pi_2$ is simply the number of lines
	(resp.\ conics) through a general point of $X$.
	Hence \cite[Lem.~2.3]{FN89}	implies that $\pi_2$ has degree $3$
	if $(\Index(X),d(X)) = (2,5)$ and
	the table given in \cite[\S 3]{IM07} implies that  $\pi_2$ has degree $6$ if
	$(\Index(X),g(X))=(1,12)$. Thus $\pi_2(V(k))$ is also thin in these cases.
\end{proof}

\subsection{Fano threefolds which are complete intersections}
We close this section by specialising to Fano complete intersections
 $X \subset \PP^n$ of dimension $3$.
Suppose that $X$ is cut out by $s$ hypersurfaces
of degrees $d_1,\dots,d_s\geq 2$ over a number field $k$.
Then a simple calculation shows that $X$
must be of one of the following types:
\begin{enumerate}
	\item a quadric in $\PP^4$,
	\item a cubic in $\PP^4$,	
	\item a quartic in $\PP^4$,	
	\item an intersection of two quadrics in $\PP^5$,
	\item an intersection of a quadric and a cubic in $\PP^5$, or 
	\item an intersection of three quadrics in $\PP^6$.	  
\end{enumerate}
A non-singular quadric in $\PP^4$
is a flag variety, so  Manin's conjecture for (1) follows from  \cite{FMT89}.
As we have  already seen, in cases (2)--(6) the variety contains saturated
subvarieties (at least after a finite field extension). Moreover, in case (4)
the saturated subvarieties can even be Zariski dense. The calculation
of the saturated subvarieties in the remaining cases is itself a non-trivial task,
as one needs to calculate the rational points on a surface of general type.
Establishing an asymptotic formula for
$ N(U,H,B)$, for some open subset $U$ of a Fano threefold of the above type, is likely to be very hard.

\section{A quadric bundle in biprojective space}\label{s:bi}
We now come to an example which was originally suggested by Colliot-Th\'{e}l\`{e}ne (see \cite[Ex.~3.5.3]{BT98}
and \cite[p.~346]{Pey03}). Let $X\subset \PP^3\times\PP^3$ be the quadric bundle given by
\begin{equation} \label{def:bundle}
	x_0y_0^2 + x_1y_1^2 +x_2y_2^2 + x_3y_3^2 = 0,
\end{equation}
viewed over a number field $k$.
Consider the associated counting function $N(U,H,B)$, for some open subset $U\subset X$, where $H$ is an anticanonical height associated to $\omega_X^{-1}=\OO_X(3,2)$.
For $\{i,j,k,\ell\}=\{1,2,3,4\}$,
the subvarieties
$x_i=x_j=x_k=y_\ell=0$ should form strongly accumulating subvarieties and we ask $U$ to avoid them all.
The Lefschetz hyperplane theorem (see \cite[Ex.~3.1.25]{Lar07}) implies that $\Pic (X) \cong \Pic (\PP^3 \times \PP^3) \cong \ZZ^2$,
whence \eqref{con} predicts that
$$
N(U,H,B)\sim c_{U,H} B\log B,  \quad \mbox{as $B\rightarrow \infty$},
$$
for  suitable  $c_{U,H}>0$. In \cite[p.~346]{Pey03}, it was suggested that
this asymptotic formula cannot hold with  Peyre's constant. As an application of Corollary \ref{cor:WA2}, we are
able to show that this is indeed the case.
 
\begin{theorem} \label{thm:bundle}
	Let $X$ be as in \eqref{def:bundle} and let $U \subset X$ be a non-empty open subset. For any $A \geq 0$,
	there exists a choice of anticanonical height function $H_A$ on $X$ such that
	$$
	\liminf_{B \to \infty} \frac{N(U,H_A,B)}{B \log B} > (1 + A) c_{H_A,Peyre}.
	$$
\end{theorem}
\begin{proof}
	We begin by recalling some basic facts about
	counting rational points on non-singular quadric surfaces $S\subset \PP^3$ defined over a number field $k$.
	We have $\rho(S) \in \{1,2\}$, with $\rho(S) = 2$ if and only if the discriminant
	of the underlying quadratic form is a square in $k$.
	We say that $S$ is {\em split} if $S(k)\neq \emptyset$ and $\rho(S) = 2$.
	Assume that $S(k) \neq \emptyset$ and let $H$ be an anticanonical height function on $S$.
	Then for any non-empty open subset $U \subset S$ we have
	$$
	N(U,H,B)\sim c_{H,\text{Peyre}} B(\log B)^{\rho(S)-1}, \quad \mbox{as $B\rightarrow \infty$}.
	$$
	This is a special case of the work of Franke, Manin and Tschinkel \cite{FMT89} on flag varieties $P\setminus G$,
	with $G$ taken to be the orthogonal group of the associated quadratic form. 

	We now come to the application of Corollary \ref{cor:WA2}.
	One may view $X$  as a family of quadric surfaces $S_x$, for $x\in \PP^3$.
	By the above, we see that the split quadrics contained in $X$ are saturated.
	Moreover, the collection of $x$ for which $S_x$ is split is Zariski dense in $\PP^3$. Indeed, $S_x$ is split if 
	$x_0,x_1,-x_2$ and $-x_3$ are all squares. Hence the split fibres are Zariski dense in $X$ over $k$ and so
	Corollary~\ref{cor:WA2} yields the result.
\end{proof}

Note that the phenomenon which occurs in the proof of Theorem \ref{thm:bundle}
is compatible with Conjecture \ref{conj:Manin}. Indeed, for any split quadric
surface $S_x$ in the family, $x_0x_1x_2x_3$ must be a square. The collection of rational points on $X$
which satisfy this condition is clearly thin.

\end{document}